\documentclass[preprint,12pt]{elsarticle2}

\usepackage{amssymb,amsfonts,amsmath,amsthm,graphicx}
\usepackage{verbatim}
\usepackage{multirow}
\usepackage[sub,ovp]{psfragx}
\usepackage{pmat}
\def\bfm#1{\boldsymbol{#1}}
\def\NN{\mathbb{N}}
\def\RR{\mathbb{R}}
\def\ZZ{\mathbb{Z}}
\def\bfm#1{\boldsymbol{#1}}

\journal{Applied Mathematics and Computation}
\newtheorem{thm}{Theorem}
\newtheorem{lem}[thm]{Lemma}
\newtheorem{conj}[thm]{Conjecture}

\newtheorem{prop}[thm]{Proposition}
\newdefinition{rmk}{Remark}
\newdefinition{dsc}[rmk]{Discussion}
\begin{document}

\begin{frontmatter}



\title{On positivity of principal minors of bivariate B\'{e}zier collocation matrix}


\author[jaklic]{Ga\v{s}per Jakli\v{c}}
\author[kanduc]{Tadej Kandu\v{c}\corref{cor}}
\address[jaklic]{FMF and IMFM, University of Ljubljana and PINT, University of Primorska, Jadranska 19, Ljubljana, Slovenia}
\address[kanduc]{Turboin\v{s}titut d.d., Rov\v{s}nikova 7, Ljubljana, Slovenia}
\cortext[cor]{Corresponding author.}
\ead{kanduc.tadej@gmail.com}
\date{August 1, 2013}

\begin{abstract}
It is well known that the bivariate polynomial interpolation problem at 
uniformly distributed domain points of a triangle is correct.
Thus the corresponding interpolation matrix $M$ is nonsingular. 
L.\,L. Schumaker stated the conjecture 
that all principal submatrices of $M$ are nonsingular too.
Furthermore, all of the corresponding determinants (the principal minors) are conjectured to be positive.
This result would solve the constrained interpolation problem.
In this paper, 
the conjecture on 
minors for polynomial degree $\le 17$ and conjecture for some particular 
configurations of domain points are confirmed.
\end{abstract}

  
\end{frontmatter}


\section{Introduction}

Positivity of determinants (or minors) of collocation matrices is important in approximation theory. Nonsingularity of a collocation matrix
implies existence and uniqueness of the solution of the associated interpolation problem. Positivity of principal minors or even total positivity is 
used in the proofs of some well-known results, see \cite{deBoor,TotalP,totalPositivity}, e.g.

Recently, nonsingularity and principal minors of collocation matrices for bivariate polynomial interpolation at Padua-like points and for interpolation by triangular 
B\'ezier patches were studied in \cite{DetConj} and \cite{Sfot}. A related problem is a construction or an approximation of Fekete points for a given
domain, i.e., the interpolation points, which yield the maximal absolute value of the Vandermonde determinant \cite{FeketePoints}.

It is well-known that bivariate \emph{Bernstein polynomials} $\{B_{ijk}^d\}_{i+j+k=d}$ of degree $d$ form a basis of the space of bivariate polynomials of degree $\leq d$. Let $\mathcal{I}_d=\{ (i,j,k):\; i,j,k \in \NN \cup \{0\},\, i+j+k=d  \}$.
For every $\binom{d+2}{2}$ points in a domain, which do not lie on an algebraic hyper-surface of degree $\leq d$, the corresponding interpolation problem is correct \cite{Lattices}. In particular, the problem is correct for uniformly distributed \emph{domain points}
\begin{align*}
\mathcal{D}_{d,T}:=\{(i/d,j/d,k/d):\; (i,j,k) \in \mathcal{I}_d \},
\end{align*} 
expressed in barycentric coordinates with respect to a triangle $T$. 
Thus the corresponding interpolation matrix $M:=[B_\eta^d(\xi)]_{\xi \in \mathcal D_{d,T},\,\eta \in \mathcal{I}_d}$ is nonsingular. 
In \cite{conj}, a theorem was stated that by choosing an arbitrary nonempty subset $\mathcal{J}\subset\mathcal D_{d,T}$ and the corresponding set of indices $\Gamma$, the submatrix $M_\Gamma:=[B_\eta^d(\xi)]_{\xi \in \mathcal{J},\,\eta \in \Gamma}$ is nonsingular for all $d\leq 7$, and furthermore, $\det M_\Gamma>0$. The authors of \cite{conj} verified the theorem by computer only. 
In \cite{Sfot}, a proof of nonsingularity of principal matrices of the collocation matrix for some special configurations of domain points is provided.

A straightforward way of verifying the conjecture by computing principal minors of $M$ is time consuming due to the exponential growth of the number of subsets that need to be analysed, and cannot be done within a reasonable time for $d>7$ using current computational facilities. For example, for $d = 10$ one would need to compute $2^{\binom{10+2}{2}} - 1 \approx 7.4 \cdot 10^{19}$ minors.

Matrices with positive principal minors are called \emph{P-matrices}. A lot of their properties are known, see \cite{P-mtr,Coxson}, e.g. Unfortunately, this theory
could not be applied for the study of the problem at hand, so a different approach, based on positive definiteness, is used.

In this paper, the conjecture on positivity of principal minors of the bivariate B\'ezier collocation matrix $M$ is confirmed for arbitrary $\Gamma\subset \mathcal I_d$ for $d\leq 17$.
Thus the corresponding constrained Lagrange interpolation problem has a unique solution.
This covers all the cases useful in practice, since it is well known that polynomials of high degrees have undesired properties.
Some particular configurations of domain points are analysed. A conjecture for exact lower bound of $\det M_\Gamma$ is stated. The paper is concluded by some remarks and comments on future work.

\section{Main results}

Let $\bfm i$ be a weak 3--composition of an integer $d$, i.e., $\bfm i= (i,j,k)$, such that $|\bfm i|:=i+j+k=d$ and $i, j, k\in\NN \cup \{0\}$. Let $\mathcal I_d:=\{\bfm i\}_{|\bfm i|=d}$ be a set of all weak 3--compositions of the integer $d$. The set $\mathcal I_d$ consists of $\binom{d+2}{2}$ compositions.

Let $T$ be a triangle in the plane $P$. Every point $\bfm{v}\in P$ can be written in barycentric coordinates $\bfm{v}=(u,v,w)$, $u+v+w=1$, with respect to $T$. The \emph{Bernstein basis polynomials} of total degree $d$ in barycentric coordinates are defined as
\begin{equation*}
B_{\bfm i}^d(\bfm v):=B_{ijk}^d(u,v,w):=\binom{d}{\bfm i}\, \bfm v ^{\bfm i}:=\frac{d!}{i!j!k!}u^iv^jw^k,\qquad |\bfm i|=d.
\end{equation*}
Here the standard multi-index notation and a convention $0^0=1$ are used.

Let us denote the subset of all compositions with $z$ zeros by $\mathcal I_d^{(z)}\subset\mathcal I_d$, $z=0,1,2$. Let $\xi_{\bfm i}:=\xi_{ijk}:=\bfm i/d$ be a \emph{domain point} of the triangle $T$, represented in barycentric coordinates. A domain point $\xi_{\bfm i}$ is boundary if at least one of its barycentric coordinates is zero, i.e., $\bfm i\in\mathcal I_d^{(1)}\cup\mathcal I_d^{(2)}$. An example of domain points is shown in Figure~\ref{fig:interpLines}.\\

\begin{figure}[!htb]
\centering
\begin{overpic}[width=6cm]{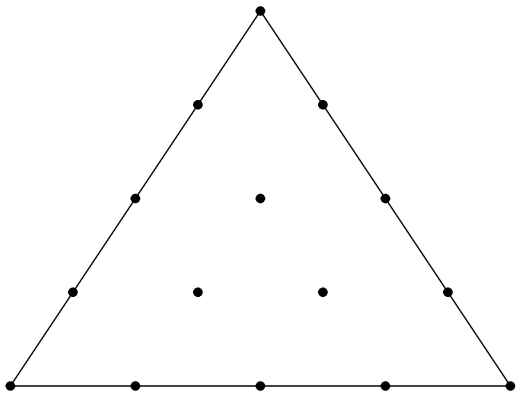}
\put(0.5,6){\small $\xi_{400}$}
\put(24.5,6){\small $\xi_{310}$}
\put(48.5,6){\small $\xi_{220}$}
\put(72.5,6){\small $\xi_{130}$}
\put(96.5,6){\small $\xi_{040}$}

\put(12.5,24){\small $\xi_{301}$}
\put(36.5,24){\small $\xi_{211}$}
\put(60.5,24){\small $\xi_{121}$}
\put(84.5,24){\small $\xi_{031}$}

\put(24.5,42){\small $\xi_{202}$}
\put(48.5,42){\small $\xi_{112}$}
\put(72.5,42){\small $\xi_{022}$}

\put(36.5,60){\small $\xi_{103}$}
\put(60.5,60){\small $\xi_{013}$}

\put(48.5,78){\small $\xi_{004}$}
\end{overpic}
\caption{Uniformly distributed domain points on triangle for $d=4$.}
\label{fig:interpLines}
\end{figure}

We are now ready to present the conjecture, stated in \cite{Sfot}, that will be tackled in this paper.

\begin{conj}[\cite{Sfot}, Conjecture 2.22]\label{conj:main}
For a given triangle $T$ and every nonempty set $\Gamma=\{\bfm i_1,\bfm i_2,\dots,\bfm i_n\}$ $\subset \mathcal I_{d}$, the matrix
\begin{align*}
M_\Gamma := [B^d_{\bfm j}(\xi_{\bfm i})]_{\bfm{i},\bfm{j}\in\Gamma}=
\left[
\begin{array}{cccc}
B_{\bfm i_1}^d(\xi_{\bfm i_1}) & B_{\bfm i_2}^d(\xi_{\bfm i_1}) & \dots & B_{\bfm i_n}^d(\xi_{\bfm i_1})\\
B_{\bfm i_1}^d(\xi_{\bfm i_2}) & B_{\bfm i_2}^d(\xi_{\bfm i_2}) & \dots & B_{\bfm i_n}^d(\xi_{\bfm i_2})\\
\vdots&\vdots&\ddots&\vdots\\
B_{\bfm i_1}^d(\xi_{\bfm i_n}) & B_{\bfm i_2}^d(\xi_{\bfm i_n}) & \dots & B_{\bfm i_n}^d(\xi_{\bfm i_n})
\end{array}
\right]
\end{align*}
is nonsingular. Furthermore, $\det M_\Gamma>0$.
\end{conj}

A confirmation of Conjecture~\ref{conj:main} would imply the following. Let $\Gamma\subset\mathcal I_d$ and let $\mathcal L(\{B_{\bfm i}^d\}_{\bfm i\in\Gamma})$ be the given interpolation space. Then the interpolation problem for the points $\{\xi_{\bfm i}\}_{\bfm i \in \Gamma}$ in the domain  
would be correct, i.e., a unique interpolant would exist. Since some control coefficients of the sought polynomial are predetermined, the problem is called \emph{constrained interpolation problem}. The conjecture is important for interpolation with spline functions, since some degrees of freedom are determined by the smoothness conditions and the rest by the interpolation conditions (see \cite{Sfot, JaklicModic}, e.g.). 

Note that the matrix $M_\Gamma$ is not symmetric. Its entries are non-negative rational numbers and the largest element of every row and column of $M_\Gamma$ are on the main diagonal of the matrix. 
Eigenvalues of $M_{\mathcal I_d}$ are derived in a closed form in \cite{Cooper_Waldron-multiBernstein},
\begin{align}\label{eqn:eigs}
\lambda_{\ell}:=\frac{d!}{(d-\ell)!\,d^\ell},\qquad \ell=1,2,\dots,d,
\end{align}
with multiplicities $3,3,4,5,\dots,d+1$.

The determinant of $M_\Gamma$ is independent of the ordering of elements of $\Gamma$ since the same ordering for rows and columns of $M_\Gamma$ is used. 
It is common to use the counter-lexicographical ordering $\succ_{\textrm{c-lex}}$,
\begin{align*}
(d,0,0),\;(d-1,1,0),\;(d-1,0,1),\;(d-2,2,0),\dots,\;(0,0,d),
\end{align*}
but a particular ordering of elements in $\mathcal I_d$, which yields a block lower triangular matrix $M_{\mathcal I_d}$,
will be more convenient \cite{CalcDet}. The linear ordering $\succ_{\rm b}$ on $\mathcal I_d$ is defined as: $\bfm i\succ_{\rm b}\bfm j$ if one of the following holds true:
\begin{enumerate}
\item $\bfm i \in \mathcal I_d^{(z_1)}\textrm{ and } \bfm j \in \mathcal I_d^{(z_2)} \textrm{ for }z_1>z_2,\;z_i\in\{0,1,2\}$,
\item $\bfm i, \bfm j \in \mathcal I_d^{(z)} \textrm{ for } z\in\{0,1,2\} \;\;\; \textrm{ and }\;\;\; \mbox{sgn}(\bfm i) \succ_{\textrm{c-lex}} \mbox{sgn}(\bfm j)$,
\item $\bfm i, \bfm j \in \mathcal I_d^{(z)} \textrm{ for } z\in\{0,1,2\} \;\;\; \textrm{ and }\;\;\; \mbox{sgn}(\bfm i) = \mbox{sgn}(\bfm j)  \;\;\; \textrm{ and } \;\;\; \bfm i \succ_{\textrm{c-lex}}\bfm j$.
\end{enumerate}
Here $\mbox{sgn}(\bfm i)=\mbox{sgn}(i,j,k):=(\mbox{sgn}(i),\mbox{sgn}(j),\mbox{sgn}(k))$ and
\begin{align*}
\mbox{sgn}(\ell):=
\left\{\begin{array}{rl}
1, & \ell>0\\
0, & \ell=0\\
-1, & \ell<0\\
\end{array}\right.
\end{align*}
is the sign function.

Let us explain the ordering $\succ_{\rm b}$ in more detail.
Firstly, the elements of $\mathcal I_d$ are ordered with respect to the number of zeros in their 3--decompositions. Then, the elements with the same number of zeros are sorted with respect to the position of the zeros. Finally, the elements that we cannot distinguish by the first two conditions are ordered counter-lexicographically.

The ordering $\succ_{\rm b}$ implies that the matrix $M_{\mathcal I_d}$ has a block lower triangular structure
\begin{align*}
M_{\mathcal I_d}=
\left[
\begin{array}{ccc}
M_{\mathcal I_d^{(2)}}&&\\
* & M_{\mathcal I_d^{(1)}}&\\
* & * & M_{\mathcal I_d^{(0)}}
\end{array}
\right],
\end{align*}
where 
\begin{align*}
&M_{\mathcal I_d^{(2)}}=
\left[
\begin{array}{ccc}
M_{\{(d,0,0)\}} & &\\
& M_{\{(d,0,0)\}} &\\
& & M_{\{(d,0,0)\}}
\end{array}
\right],
&M_{\mathcal I_d^{(1)}}=
\left[
\begin{array}{ccc}
M_{\epsilon} & &\\
& M_{\epsilon} &\\
& & M_{\epsilon}
\end{array}
\right],
\end{align*}
and
$\epsilon:=\{\xi_{ij0}\in\mathcal I_d:\; i,j\geq 1\}$.

From the structure of the matrix  $M_{\mathcal I_d}$ it follows that the problem of verifying the positivity of principal minors of the matrix is reduced to each diagonal block matrix separately. The first three $1\times 1$ blocks correspond to domain points at the vertices of the triangle $T$. The next three blocks are $(d-1) \times (d-1)$ matrices $M_\epsilon$ that correspond to domain points lying in the interior of the triangle edges. The matrix $M_\epsilon$ is a univariate B\'{e}zier collocation matrix and by \cite{deBoor_deVore} it is totally non-negative with positive principal minors. Therefore, the interpolation problem reduces to the study of the remaining $\binom{d-1}{2}\times\binom{d-1}{2}$ matrix $M_{\mathcal I_d^{(0)}}$ that corresponds to interior domain points. Unfortunately, this block represents a very large part of the matrix $M_{\mathcal I_d}$ for a large $d$.
As an example, the matrix $M_{\mathcal I_4}$, is shown in Figure~\ref{fig:mtrM4}.

\begin{figure}[!hbt]
\begin{align*}
M_{\mathcal I_4} = \frac{1}{256}
\begin{pmat}[{|||..|..|..|..}]
256 & 0 & 0 & 0 & 0 & 0 & 0 & 0 & 0 & 0 & 0 & 0 & 0 & 0 & 0 \cr\-
0& 256& 0& 0& 0& 0& 0& 0& 0& 0& 0& 0& 0& 0& 0 \cr\-
0& 0& 256& 0& 0& 0& 0& 0& 0& 0& 0& 0&  0& 0& 0 \cr\-
81& 1& 0& 108& 54& 12& 0& 0& 0& 0& 0& 0& 0& 0& 0 \cr
16& 16& 0& 64& 96& 64& 0& 0& 0& 0& 0& 0& 0& 0& 0 \cr
1& 81& 0& 12& 54& 108& 0& 0& 0& 0& 0& 0& 0& 0& 0 \cr\-
81& 0& 1& 0& 0& 0& 108& 54& 12& 0& 0& 0& 0& 0& 0 \cr
16& 0& 16& 0& 0& 0& 64& 96& 64& 0& 0& 0& 0& 0& 0 \cr
1& 0& 81& 0& 0& 0& 12& 54& 108& 0& 0& 0& 0& 0& 0 \cr\-
0& 81& 1& 0& 0& 0& 0& 0& 0& 108& 54& 12& 0& 0& 0 \cr
0& 16& 16& 0& 0& 0& 0& 0& 0& 64& 96& 64& 0& 0& 0 \cr
0& 1& 81& 0& 0& 0& 0& 0& 0& 12& 54& 108& 0& 0& 0 \cr\-
16& 1& 1& 32& 24& 8& 32& 24& 8& 4& 6& 4& 48& 24& 24 \cr
1& 16& 1& 8& 24& 32& 4& 6& 4& 32& 24& 8& 24& 48& 24 \cr
1& 1& 16& 4& 6& 4& 8& 24& 32& 8& 24& 32& 24& 24& 48 \cr
\end{pmat}
\end{align*}
\caption{Matrix $M_{\mathcal I_4}$ with the linear ordering $\succ_{\rm b}$.}
\label{fig:mtrM4}
\end{figure}

Now we are ready to present one of the main results of the paper.

\begin{thm}\label{thm:conj17}
Let $d\leq 17$. Then Conjecture \ref{conj:main} holds true, i.e., $\det M_\Gamma>0$ for every nonempty subset $\Gamma\subset \mathcal I_d$.
\end{thm}

\begin{proof}
Every non-symmetric matrix $M$ is positive definite, i.e., $\bfm{x}^T M\bfm{x}>0$ for all 
$\bfm{x} \in \RR^{n}\backslash \{\bfm 0\}$, $n:=\binom{d+2}{2}$, iff the symmetric matrix $M + M^T$ is positive definite \cite{Johnson_PositDefMtr}. Hence the positive definiteness of $M$ can be verified by analysing the later symmetric matrix.

Fix $d$, $1\leq d\leq 16$. By using a symbolic computational toolbox we can compute Cholesky decomposition of $M_{\mathcal I_d}+M_{\mathcal I_d}^T$ in exact arithmetics. It can be verified that the decomposition exists for $1\leq d\leq 16$. Therefore the matrix $M_{\mathcal I_d}+M_{\mathcal I_d}^T$ is positive definite and so is $M_{\mathcal I_d}$. 
Thus all principal submatrices $M_\Gamma$ of $M_{\mathcal I_d}$ are positive definite too. Hence, all real eigenvalues of every $M_\Gamma$ are positive. Since the determinant of a matrix is the product of its eigenvalues, it follows that all principal minors of $M_{\mathcal I_d}$ are positive.

For $d=17$, the matrix $M_{\mathcal I_d}+M_{\mathcal I_d}^T$ has three negative eigenvalues. 
The matrices $M_{\mathcal I_d^{(2)}}$ and $M_{\mathcal I_d^{(1)}}$ are P-matrices, thus the problem reduces to the study of the matrix $M_{\mathcal I_d^{(0)}}$. 
It can be verified that the Cholesky decomposition of the matrix $M_{\mathcal I_d^{(0)}}+M_{\mathcal I_d^{(0)}}^T$ exists and the rest of the proof is similar to the first part.
\end{proof}

Note that instead of computing the Cholesky decomposition of $M_{\mathcal I_d}+M_{\mathcal I_d}^T$ it is equivalent to verify that all of $\binom{d+2}{2}$ leading principal minors or all of $\binom{d+2}{2}$ eigenvalues are positive. The exact expressions are too large to be written in the paper but the reader can easily verify the results.

For $d=18$, the smallest eigenvalue of $M_{\mathcal I_d}+M_{\mathcal I_d}^T$ is approximately $-1.1 \cdot 10^{-7}$.
For $d \ge 18$, the number of negative eigenvalues of the matrix $M_{\mathcal I_d}+M_{\mathcal I_d}^T$ increases with $d$.
Therefore, this approach cannot be used to prove the conjecture in general, since the positive definiteness is only a sufficient condition for positivity of principal minors.
However, our result covers all the cases important in practice, since only Lagrange polynomial interpolants of low degrees are useful. 

The reason why the presented approach only works for $d\leq 16$ for $M_{\mathcal I_d}$ can be explained by the following observation. Since the matrix is not symmetric, its numerical range $\{(\bfm x^T M_{\mathcal I_d} \bfm x)/(\bfm x^T \bfm x):\; \bfm x \in \RR^n\backslash \{\bfm 0\} \}$, $n:=\binom{d+2}{2}$, is not bounded by the largest and the smallest eigenvalue of $M_{\mathcal I_d}$. Eigenvalues of $M_{\mathcal I_d}$ are positive (see \eqref{eqn:eigs}) but the majority of them are closer and closer to zero as we increase $d$. At $d=17$, the eigenvalues are so dense around zero that the numerical range passes the zero border. At that point the matrix is no longer positive definite.

Theorem~\ref{cor:constInterpol} confirms the second part of the Conjecture 2.22 in \cite{Sfot} for $d\leq 17$. The result follows straightforwardly from Theorem~\ref{thm:conj17}.

\begin{thm}\label{cor:constInterpol}
Let $\Gamma\subset \mathcal I_d$ and let $d\leq 17$. Then for any $\{z_{\bfm i} \in \RR \}_{\bfm i\in\Gamma}$, there is a unique polynomial of the form
\begin{align*}
p:=\sum_{\bfm i \in \Gamma}c_{\bfm i} B^d_{\bfm i}
\end{align*}
that solves the constrained interpolation problem
\begin{align*}
p(\xi_{\bfm i})=z_{\bfm i},\qquad \bfm i\in\Gamma.
\end{align*}
\end{thm}

\begin{rmk}
Theorem~\ref{cor:constInterpol} generalizes \cite[Theorem~3]{conj}. Its proof avoids computation of all subdeterminants, as was the case in \cite{conj}.
\end{rmk}

\begin{rmk}\label{rmk:trivar}
The approach for proving Theorem~\ref{thm:conj17} can also be used to prove a similar result for the trivariate case: the trivariate B\'ezier collocation
matrix for $d \le 15$ is positive definite, and thus the corresponding constrained polynomial Lagrange interpolation problem is correct.
\end{rmk}

Let us simplify the considered matrix $M_{\Gamma}$. Let us construct a matrix $N_\Gamma$ from $M_\Gamma$ in the following way:
\begin{itemize}
\item for every column, divide each element, that corresponds to a polynomial $B_{\bfm i}^d$, by $\binom{d}{\bfm i}$,
\item multiply the obtained matrix by $d^d$.
\end{itemize}
Thus an element $B_{\bfm{j}}^d(\xi_{\bfm{i}})$ is transformed into $\bfm{i}^{\bfm{j}}$.
Therefore
\begin{equation} \label{eqn:NM}
\det N_\Gamma = \frac{d^{d\cdot |\Gamma|}}{\displaystyle{\prod_{\bfm i\in\Gamma} \binom{d}{\bfm i}}} \det M_\Gamma.
\end{equation}

Clearly, the matrix $N_\Gamma$ is a principal submatrix of $N_{\mathcal I_d}$. Since $\mbox{sgn} (\det M_\Gamma) = \mbox{sgn} (\det N_\Gamma$), Conjecture \ref{conj:main} holds true for $M_\Gamma$ iff it holds true for $N_\Gamma$.

The matrix $N_\Gamma$ has some nice properties. It consists only of non-negative integers, thus determinant computations are exact. The matrix $N_{\mathcal I_d}$ has a simpler structure than $M_{\mathcal I_d}$ and is closely related to combinatorial objects. Therefore some properties of the matrix 
$M_{\mathcal I_d}$ will be proven via $N_{\mathcal I_d}$.
Note that some of the properties are not preserved by the transformation $M_\Gamma\to N_\Gamma$ (for example, see Remark~\ref{rmk:N4}).

\begin{rmk}\label{rmk:N4}
The matrix $N_{\mathcal I_d}$ is positive definite only for $d\leq 4$.
\end{rmk}

Now let us consider some particular configurations of domain points (and the corresponding choices of $\Gamma$) for $d$ arbitrary.

\begin{thm} \label{thm:Gamme}
Let $d$ be arbitrary and let $\Gamma$ satisfy one of the following assumptions:
\begin{enumerate}
\item $|\Gamma|\leq 2$,
\item let one of the components of $(i,j,k)$ be fixed for all elements in $\Gamma$,
\item $\Gamma=\mathcal I_d$,
\item $\Gamma=\{(i,j,k)\in \mathcal I_d: \; i\geq i_0,\, j \geq j_0,\, k \geq k_0\}$ for fixed non-negative integers $i_0,\, j_0,\, k_0$,
\item $\Gamma \subset \mathcal I_d^{(2)} \cup \mathcal I_d^{(1)}$,
\item $\Gamma=\Gamma_1\cup\Gamma_2$, where $\Gamma_1$ is one of the sets, defined in \emph{1.}, \emph{2.}\ or \emph{4.}, and $\Gamma_2$ is a set in \emph{5.}
\end{enumerate}
Then $\det M_{\Gamma}>0$.
\end{thm}

\begin{proof}
Case 1: For $|\Gamma|=1$, the matrix $M_\Gamma$ is a positive number. 

Now let $\Gamma=\{\bfm i_1, \bfm i_2\}$. Since the largest element of every column in $M_\Gamma$ is on the diagonal of $M_\Gamma$,
\begin{align*}
\det M_\Gamma =
\left|
\begin{array}{cc}
B_{\bfm i_1}^d(\xi_{\bfm i_1}) & B_{\bfm i_2}^d(\xi_{\bfm i_1}) \\
B_{\bfm i_1}^d(\xi_{\bfm i_2}) & B_{\bfm i_2}^d(\xi_{\bfm i_2})
\end{array}
\right|>0.
\end{align*}
\\

Case 2: Let one of the components of $(i,j,k)$ be fixed.
Without loss of generality we may assume that $\bfm i_\ell=(i_\ell, j_\ell, k)$, $i_\ell+j_\ell+k=d$, $\ell \in \{1,2,\dots,|\Gamma|\}$. By dividing each element of $N_\Gamma$ by $k^k$ and multiplying each column by a proper constant, the matrix $N_\Gamma$ transforms to a univariate B\'{e}zier collocation matrix, which is a P-matrix by \cite{deBoor_deVore}.\\

Case 3: It follows from \cite{CalcDet} and (\ref{eqn:NM}) that
\begin{align}\label{eqn:detMId}
\det M_{\mathcal I_d}=
d^{-d\binom{d+2}{2}}\prod_{\bfm i\in\mathcal I_d} \binom{d}{\bfm i}
\prod_{k=1}^{\min\{d,3\}}\left( d^{\binom{d-1}{k}}\prod_{i=1}^{d-k+1} i^{(d-i+1)\binom{d-i-1}{k-2}} \right)^{\binom{3}{k}}>0.
\end{align}

Case 4: Firstly, let us revise the proof of nonsingularity of $M_\Gamma$ (see \cite{Sfot}). By appropriately multiplying rows and columns of $M_\Gamma$, we obtain a collocation matrix $\tilde M_\Gamma$ consisting of all polynomials of total degree $\leq d_0:=d-i_0-j_0-k_0$ and domain points that correspond to $\Gamma$. Note that the interpolation problem remains correct if the domain points $\{\xi_{\bfm i}:\; \bfm i \in \mathcal I_{d_0}\}$ are translated and scaled by a positive factor (see \cite{Chui_Lai-Vandermonde-Lagrange} or Theorem 1.10 in \cite{Sfot}). Let $M(\lambda)$, $\lambda\in[0,1]$, denote a homotopy that changes the domain points by such transformation and $M(0)=M_{\mathcal I_{d_0}}$, $M(1)=\tilde M_\Gamma$. Since $\det M_{\mathcal I_{d_0}}>0$ (see case 3 of this theorem) and the matrix $M(\lambda)$ is nonsingular for every $\lambda\in[0,1]$, it follows that $\det \tilde M_\Gamma>0$.\\

Case 5, case 6: For $\Gamma \subset \mathcal I_d^{(2)} \cup \mathcal I_d^{(1)}$ and $\Gamma=\Gamma_1\cup\Gamma_2$, the result follows straightforwardly from the structure of the matrix $M_\Gamma$.
\end{proof}

\begin{rmk}
The set $\Gamma$ in Theorem~\ref{thm:Gamme}, case 2 corresponds to domain points in the triangle $T$, lying on a line parallel to some edge of $T$.

Although the proof that $\det M_{\mathcal I_d}>0$ may seem easy, most of the paper \cite{CalcDet} is dedicated to the derivation of determinant formula in a closed form (\ref{eqn:detMId}).

The set $\Gamma$ in case 4 corresponds to domain points that form a scaled triangle of triangle $T$.

The subset of decompositions $\Gamma$ in case 5 corresponds to the constrained interpolation problem at boundary domain points of the triangle $T$.

\end{rmk}

\section{New conjectures and open problems}

In the rest of the paper we state some conjectures that expand the Conjecture \ref{conj:main} 
and present a small interpolation problem with three domain points that may seem easy at first glance but the existence of the solution remains unproven so far.

The following two conjectures state exact lower bounds for $\det M_\Gamma$ and $\det N_\Gamma$, respectively.
\begin{conj}\label{conj-mindet}
For $d$ fixed,
\begin{equation*} 
\min _{\substack{\Gamma\subset{\mathcal I_d}\\ \Gamma\neq \emptyset}} \det M_\Gamma= \det M_{\mathcal I_d},
\end{equation*}
where $\det M_{\mathcal I_d}$ is given in \eqref{eqn:detMId}.
\end{conj}


\begin{conj}\label{conj:lowN} For $\ell\in\NN$ let
\begin{align*}
n_d:=
\left\{
\begin{array}{cl}
\ell^{3\ell}, & d=3\ell\\
(\ell+1)^{\ell+1}\ell^{2\ell}, & d=3\ell+1\\
(\ell+1)^{2\ell+2}\ell^\ell, & d=3\ell+2
\end{array}
\right..
\end{align*}
Then
\begin{equation*} 
\min _{\substack{\Gamma\subset{\mathcal I_d}\\ \Gamma\neq \emptyset}} \det N_\Gamma= n_d.
\end{equation*}
\end{conj}

Conjectures~\ref{conj-mindet} and \ref{conj:lowN} were verified by a computer for $d\leq 7$. Let us prove the latter conjecture for $|\Gamma| \leq 2$ and arbitrary $d$. We will need the following lemma.

\begin{lem} \label{lemma:max}
Let $\bfm x=(x,y,z)\in\RR^3$ and fix $\bfm i=(i,j,k)\in \mathcal I_d$. Let the function $f(\bfm x):=\bfm x^{\bfm i}$ be defined on triangle
\begin{align}\label{eqn:obm}
\Omega :=\{\bfm x: \; x+y+z=d,\, 0\leq x\leq d,\, 0\leq y\leq d-x\} \subset \RR^3.
\end{align}
Then $f$ has a unique maximum at $\bfm i$ and
\begin{align*}
\max_{\bfm x\in\Omega} f(\bfm x) = \bfm i^{\bfm i}.
\end{align*}
\end{lem}

\begin{proof}
Let $T=\langle(0,0),(d,0),(0,d) \rangle \subset \RR^2$ be a triangle in the domain and let us define Bernstein polynomial $B_{\bfm i}^d$ on $T$. By interpreting barycentric coordinates of $B_{\bfm i}^d$ as points in $\RR^3$,  $B_{\bfm i}^d(\bfm x/d) = \binom{d}{\bfm i}/d^d f(\bfm x)$. Since $B_{\bfm i}^d$ has a unique maximum in the barycentric point $\xi_{\bfm i}=\bfm i/d$, the proof is complete.
\end{proof}

\begin{prop}
Conjecture \ref{conj:lowN} holds true for $|\Gamma|\leq 2$.
\end{prop}

\begin{proof} Let $|\Gamma| =1$ and let $g(\bfm x):=g(x,y,z):=\bfm x^{\bfm x}$ be a function, defined on the domain $\Omega$ as in (\ref{eqn:obm}). 
We are looking for
\begin{align*}
\alpha:=\min_{\bfm x\in \Omega\cap\ZZ^3} g(\bfm x).
\end{align*}

A unique local minimum of $g$ in the interior of $\Omega$ is obtained as the solution of the normal system 
$\partial g/\partial x=0,\, \partial g/\partial y =0$, and it is reached at $(d/3,d/3,d/3)$. This is a global minimum since $g(d/3,d/3,d/3)< g(\bfm x)$ for all $\bfm x$ at the boundary of $\Omega$.

If $d\equiv 0\, (\textrm{mod }3)$, then $\alpha=(d/3)^d$.

Let us examine the case $d\equiv 1 (\textrm{mod }3)$. Then $d=(\ell+1)+\ell+\ell$ for $\ell\in\NN\cup\{0\}$. By the symmetry
of the function $g$ and since at least one component of $\bfm x\in\Omega\cap \ZZ^3$ is greater or equal to $\ell+1$, it is enough to consider the case $\ell+1\leq x$ only. 

Let us define 
\begin{align*}
\Omega_x:=\{\bfm{x}=(x,y,z)\in \Omega:\;\ell+1\leq x\}.
\end{align*}

Since $g$ has no extreme point in $\Omega_x$, the minimum value is reached at the boundary of $\Omega_x$. Then the minimum is $\alpha=(\ell+1)^{\ell+1}\ell^{2\ell}$ and it is achieved at $(\ell+1,\ell,\ell)$. For $\ell+1\leq y$ and $\ell+1\leq z$, the derivation is analogous.

The case $d\equiv 2\, (\textrm{mod }3)$ is similar to the previous one. Since $N_{\{\bfm i\}}=g(\bfm i)$, $\bfm i\in\mathcal I_d$, and $n_d=\alpha$, the conjecture for $|\Gamma|=1$ is proven.\\

Now let us consider the case $|\Gamma|=2$. 
Let us show that
\begin{equation*} 
\bfm i_1^{\bfm i_1}=\det N_{\{\bfm i_1\}} \leq \det N_{\{\bfm i_1,\, \bfm i_2\}}=\bfm i_1^{\bfm i_1} \bfm i_2^{\bfm i_2}-\bfm i_2^{\bfm i_1} \bfm i_1^{\bfm i_2}
\end{equation*}
for every $\bfm i_1,\,\bfm i_2\in \mathcal I_d$, $\bfm i_1\neq \bfm i_2$.

By Lemma~\ref{lemma:max} it follows that $\bfm i_2^{\bfm i_1} < \bfm i_1^{\bfm i_1}$ and $\bfm i_1^{\bfm i_2} \leq \bfm i_2^{\bfm i_2}-1$, thus
\begin{align*}
\det N_{\{\bfm i_1,\, \bfm i_2\}} - \det N_{\{\bfm i_1\}} = \bfm i_1^{\bfm i_1}(\bfm i_2^{\bfm i_2}-1)-\bfm i_2^{\bfm i_1}\bfm i_1^{\bfm i_2} \geq 0.
\end{align*}
\end{proof}

One of the most straightforward approaches to prove the Conjecture~\ref{conj:main} would be by mathematical induction with respect to the number of domain points and $d$ arbitrary. For example, when $|\Gamma| \leq 2$ the conjecture is easily verified (see case 1 of Theorem~\ref{thm:Gamme}).  On the other hand, the case $|\Gamma|=3$ is considerably more complex.

To present entries in the matrix $N_\Gamma$ for $|\Gamma|=3$ better, we use a slightly different notation here. Let $n\in \NN$ and let $\Gamma:=\{(a,b,c),\,(d,e,f),\,(g,h,i)\} \subset \mathcal I_n$ be a set of three pairwise different weak 3--compositions. We would need to show that the determinant of $N_\Gamma$,
\begin{align*}
\left|
\begin{array}{ccc}
a^a b^b c^c & a^d b^e c^f & a^g b^h c^i\\
d^a e^b f^c & d^d e^e f^f & d^g e^h f^i\\
g^a h^b i^c & g^d h^e i^f & g^g h^h i^i
\end{array}
\right|
&= a^a b^b c^c \cdot d^d e^e f^f \cdot  g^g h^h i^i
- a^a b^b c^c \cdot g^d h^e i^f \cdot  d^g e^h f^i\\
&- g^a h^b i^c \cdot d^d e^e f^f \cdot  a^g b^h c^i
- d^a e^b f^c \cdot a^d b^e c^f \cdot  g^g h^h i^i\\
&+ d^a e^b f^c \cdot g^d h^e i^f \cdot  a^g b^h c^i
+ g^a h^b i^c \cdot a^d b^e c^f \cdot  d^g e^h f^i,
\end{align*}
is positive for arbitrary configurations of three elements. Since it is difficult to quantify the sizes of the entries in the determinant and to compare them for a general case, positivity of the determinant remains unproven.

\section{Conclusions and future work}

The presented results on positivity of the principal minors provide an important theoretical background for solving constrained Lagrange interpolation problems on a triangle. The latter is a significant step to construct Lagrange interpolation splines on triangulations. Although the positivity of the minors for uniformly distributed domain points remains an open problem for polynomial degree $\geq 18$, the verified properties are satisfactory for practical applications. Namely, the use of polynomials of higher degrees is not recommended due to their tendency to oscillate. 
To prove the conjecture for general $d$, our approach with positive definiteness fails since the condition is only sufficient but not necessary.

Other approaches to tackle the conjecture are by using mathematical induction or by examining some special configurations of points regardless of the polynomial dimension $d$. The latter technique is used in Theorem~\ref{thm:Gamme} where some simple configurations of points are analysed.
If $\Gamma\subset \mathcal I_d$ does not meet very strict conditions, the analysis of the determinant of $M_\Gamma$ for arbitrary $d$ is cumbersome. This can be observed from a small example from the previous section where $|\Gamma|=3$ and $d$ is arbitrary.\\

A possible future work is to try to tackle the main conjecture and newly derived conjectures from Section 3. Obtaining some partial  results, new properties of the studied collocation matrices or developing new mathematical tools to tackle the problems would also be helpful.

The Lagrange problem can be extended by studying generalised domain points. Therefore the points need not be uniformly distributed anymore. In those cases, the question of existence and uniqueness of the interpolation problem can be upgraded by studying the optimal positions of interpolation points (for example, the points that give the smallest maximum error of the interpolation polynomial).

In the paper we only investigated bivariate collocation matrices. The main conjecture can be straightforwardly generalised to trivariate case or to higher dimensions. Tackling the conjecture in higher dimension remains a possible work in the future. 



\section{Acknowledgements}
The first author thanks ARRS for a grant P1-294. Operation part financed by the European Union, European Social Fund. We would also like to thank Simon Foucart for acquainting us with some of the properties of Bernstein matrices.

\clearpage

\end{document}